\documentclass[12pt]{amsart}

\usepackage{amssymb}
\usepackage{bm}
\usepackage{graphicx}
\usepackage{psfrag}
\usepackage{hyperref}
\hypersetup{colorlinks=true, linkcolor=blue, citecolor=red, urlcolor=wine}
\usepackage{color}
\usepackage{url}
\usepackage{algpseudocode}
\usepackage{fancyhdr}
\usepackage{xy}
\input xy
\xyoption{all}
\usepackage{stmaryrd}
\usepackage{calrsfs}
\usepackage[dvipsnames]{xcolor}
\usepackage{tikz-cd}

\usepackage{mathrsfs}

\newtheorem*{maintheorem*}{Main Theorem}
\newtheorem{theorem}{Theorem}[section]

\newtheorem{question}[theorem]{Question}

\theoremstyle{definition}

\newtheorem{remark}[theorem]{Remark}
\newtheorem{example}[theorem]{Example}
\numberwithin{equation}{section}

\newcommand{\nn}{\mathbb{N}}
\newcommand{\pp}{\mathbb{P}}
\newcommand{\qq}{\mathbb{Q}}
\newcommand{\rr}{\mathbb{R}}
\newcommand{\uu}{\mathcal{U}}
\newcommand{\zz}{\mathbb{Z}}

\providecommand\ldb{\llbracket}
\providecommand\rdb{\rrbracket}

\keywords{Puiseux monoid, internal sum, atomicity, divisibility, chain of ideals}


\begin{document}
	\mbox{}
	\title{On the internal sum of Puiseux monoids}
	
	\author{Jonathan Du}
	\address{MIT PRIMES\\MIT\\Cambridge, MA 02139}
	\email{jonathan.cx.du@gmail.com}
	
	
	\author{Bryan Li}
	\address{MIT PRIMES\\MIT\\Cambridge, MA 02139}
	\email{wowo2888@gmail.com}

	\author{Shaohuan Zhang}
	\address{MIT PRIMES\\MIT\\Cambridge, MA 02139}
	\email{szhang26@cranbrook.edu}
	
	\date{\today}
	
	\begin{abstract}
		In this paper, we investigate the internal (finite) sum of submonoids of rank-$1$ torsion-free abelian groups. These submonoids, when not groups, are isomorphic to nontrivial submonoids of the nonnegative cone of $\qq$, known as Puiseux monoids, and have been actively studied during the last few years. Here we study how the atomicity and arithmetic of Puiseux monoids behave under their internal (finite) sum inside the abelian group~$\qq$. We study the factorization properties of such internal sums, giving priority to Cohn's notion of atomicity and the classical bounded and finite factorization properties introduced and studied in 1990 by Anderson, Anderson, and Zafrullah in the setting of integral domains, and then generalized by Halter-Koch to commutative monoids. We pay special attention to how each of the considered properties behaves under the internal sum of a Puiseux monoid with a finitely generated Puiseux monoid. Throughout the paper, we also discuss examples showing that our primary results do not hold for submonoids of torsion-free abelian groups with rank larger than~$1$.
	\end{abstract}

\maketitle

\section{Introduction}
\label{sec:intro}

Let $G$ be an abelian group (additively written). The internal sum of two submonoids $S$ and $T$ of $G$, denoted by $S+T$, is the smallest submonoid of $G$ containing the set $S \cup T$. The following is an explicit way to write the same internal sum:
\[
    S+T = \{s+t : s \in S \text{ and } t \in T\}.
\]
It is clear that the property of being finitely generated is preserved under taking internal sums. Therefore, it is natural to wonder whether inside the group $G$ (used as a universe), the internal sum with a finitely generated submonoid of $G$ preserves a given property. More formally, we have the following question.

\begin{question} \label{quest:main question}
    Given an algebraic property $\mathfrak{p}$ and submonoids $M$ and $N$ of $G$ such that $M$ satisfies the property $\mathfrak{p}$ and $N$ is finitely generated, does the internal sum $M+N$ satisfy $\mathfrak{p}$?
\end{question}

The answer to Question~\ref{quest:main question} heavily depends on the property~$\mathfrak{p}$ and the universe~$G$. The properties we consider here are precisely those studied in the landmark paper~\cite{AAZ90}, where Anderson, Anderson, and Zafrullah introduced the bounded and finite factorization properties in the class of atomic domains and proposed Diagram~\ref{diag:AAZ's chain for monoids} as a methodology for a systematic investigation of the deviation of a given integral domain from being a UFD (Diagram~\ref{diag:AAZ's chain for monoids} is adapted for the class of cancellative and commutative monoids, which is larger than the class of integral domains). We also consider two further properties that are closely connected to the property of being atomic and the property of finite factorization: these are strong atomicity and the length-finite factorization property, which we define later.
\begin{equation} \label{diag:AAZ's chain for monoids}
	\begin{tikzcd}
		\textbf{ UFM } \ \arrow[r, Rightarrow]  \arrow[d, Rightarrow] & \ \textbf{ HFM } \arrow[d, Rightarrow] \\
		\textbf{ FFM } \ \arrow[r, Rightarrow] & \ \textbf{ BFM } \arrow[r, Rightarrow]  & \textbf{ ACCP monoid}  \arrow[r, Rightarrow] & \textbf{ atomic monoid}
	\end{tikzcd}
\end{equation}
\smallskip
The primary purpose of this paper is to answer Question~\ref{quest:main question} for the factorization properties in Diagram~\ref{diag:AAZ's chain for monoids} and considering rank-$1$ torsion-free abelian groups as our universe. It is well known that every rank-$1$ torsion-free abelian group is isomorphic to a subgroup of $\qq$ (see \cite[Section~18]{lF70}), and this allows us to fix the abelian group $\qq$ as our universe. By virtue of \cite[Theorem 3.12]{GGT21}, every submonoid of $\qq$ that is not a group is isomorphic to a submonoid of the nonnegative cone of $\qq$. Submonoids of $\qq_{\ge 0}$ are known as Puiseux monoids, and they have been actively investigated during the last few years, mostly in the setting of atomicity and factorization theory.
\smallskip

It is worth emphasizing that the abelian group $\qq$ is perhaps the simplest universe where the algebraic properties we are interested in are not trivially preserved by taking internal sum with a finitely generated submonoid. This statement is justified by the following two facts. First, the internal sum of two submonoids of $\zz$ is either a group or a numerical monoid, and so such an internal sum automatically satisfies all the properties in Diagram~\ref{diag:AAZ's chain for monoids} (with the exception of UFM, which is satisfied if and only if the internal sum is a cyclic monoid or a subgroup of $\zz$). Second, every submonoid of a finite abelian group is a subgroup, and so the internal sum of any two submonoids inside a finite abelian group is a group, and so it is trivially a UFM. The other potential simple universes that one may consider in Question~\ref{quest:main question} are the free abelian groups $\zz^n$ (for $n \ge 2$), and throughout this paper we discuss examples to illustrate that the answers to Question~\ref{quest:main question} for the properties we are interested in here are negative even for submonoids of $\zz^2$.
\smallskip


Although we will give a positive answer to Question~\ref{quest:main question} for most of the properties in Diagram~\ref{diag:AAZ's chain for monoids} (except for UFMs and HFMs), we should mention that there are important algebraic properties for which the corresponding answers to Question~\ref{quest:main question} are negative even inside the universe $\qq$. Perhaps two of the most important properties are the property of being a UFM (the first property in Diagram~\ref{diag:AAZ's chain for monoids}) and the property of being a Krull monoid (perhaps the property most systematically investigated in the literature of arithmetic and factorization theory); indeed, even though $2\nn_0$ and $3\nn_0$ are UFMs (and thus, Krull monoids), their internal sum, $\nn_0 \setminus \{1\}$, is not even Krull (the Puiseux monoids that are Krull monoids are precisely the cyclic submonoids of $\qq$ \cite[Corollary 6.7]{fG18}). It was also recently noted in \cite[Example~3.4]{GG24} that the internal sum of Puiseux monoids satisfying the ascending chain condition on principal ideals (ACCP) may not satisfy the ACCP.
\smallskip

Let $M$ be a submonoid of an (additive) abelian group. Following Cohn~\cite{pC68}, we say that $M$ is \emph{atomic} if every non-invertible element of $M$ can be written as a sum of atoms (i.e., irreducible elements). Also, $M$ is called \emph{strongly atomic} if $M$ is atomic and any two element of $M$ have a maximal common divisor. In Section~\ref{sec:atomic properties}, we prove that inside any rank-$1$ torsion-free abelian group, the internal sum of an atomic (resp., a strongly atomic) monoid and a finitely generated monoid is an atomic (resp., strongly atomic) monoid. This is the main result of Section~\ref{sec:atomic properties}, and we illustrate that this result cannot be generalized to submonoids of torsion-free abelian groups with higher rank. We also show that, even if we take a rank-$1$ torsion-free abelian group as our universe, the internal sum of atomic submonoids may be far from being atomic; indeed, it may contain no atoms (without being a group).
\smallskip

Now assume that the submonoid $M$ is atomic. Following Anderson, Anderson, and Zafrullah~\cite{AAZ90} and Halter-Koch~\cite{fHK92}, we say that $M$ is a \emph{finite factorization monoid} (FFM) if every non-invertible element of $M$ has a finitely many factorizations (i.e., can be written as a sum of atoms in essentially finitely many ways) while we say that $M$ is a \emph{bounded factorization monoid} (BFM) if for each non-invertible element of $M$ there exists a positive upper bound such that the number of atoms in any factorization of such element (counting repetitions) is below that bound. Following Geroldinger and Zhong~\cite{GZ21}, we say that $M$ is a \emph{length-finite factorization monoid} (LFFM) if for every positive integer $\ell$, any non-invertible element of $M$ has only finitely many factorizations with exactly $\ell$ atoms (counting repetitions). In Section~\ref{sec:factorization properties}, we prove that inside any rank-$1$ torsion free abelian group, the internal sum of an FFM (resp., a BFM, an LFFM) and a finitely generated monoid is again an FFM (resp., a BFM, an LFFM). We also illustrate that our theorem cannot be generalized to abelian groups with higher rank, and we provide examples of two Puiseux monoids that are BFMs whose internal sum is not even atomic, as well as an example of two rank-2 monoids that are FFMs whose internal sum is not atomic.

\bigskip
\section{Background}
\label{sec:background}

As it is customary, we let $\zz$, $\qq$, and $\rr$ denote the set of integers, rational numbers, and real numbers, respectively. In addition, we let $\pp$, $\nn$, and $\nn_0$ denote the set of primes, positive integers, and nonnegative integers, respectively. For $b,c \in \zz$, we denote the discrete closed interval from $b$ to $c$ by $\ldb b,c \rdb$; that is
\[
	\ldb b,c \rdb := \{n \in \zz : b \le n \le c\}.
\]
If $q \in \qq \setminus \{0\}$, then $\mathsf{n}(q)$ and $\mathsf{d}(q)$ are, respectively, the unique $n \in \zz$ and $d \in \nn$ such that $q = \frac{n}d$ and $\gcd(n,d) = 1$. For a nonzero $n \in \zz$ and $p \in \pp$, we let $v_p(n)$ denote the $p$-adic valuation of $n$, $v_p(n) = \max \{m \in \nn_0 : p^m \mid n \}$. Then, for $p \in \pp$, the $p$-adic valuation map $v_p \colon \qq \setminus \{0\} \to \zz$ is the function defined by the assignment $q \mapsto v_p(\mathsf{n}(q)) - v_p(\textsf{d}(q))$ for each $q \in \qq \setminus \{0\}$.
\smallskip

Given that all semigroups we consider here are cancellative and commutative, throughout this paper, a \emph{monoid}\footnote{The most standard definition of a monoid does not assume either cancellation or commutativity.} is a semigroup with an identity element. Also, unless otherwise specified, monoids here will be denoted additively with identity element denoted by $0$, which we will refer to as the zero element throught this paper. We assume that submonoids inherit identity elements and also that monoid homomorphisms respect identity elements. Let $M$ be a commutative monoid with identity element~$0$. We let $M^\bullet$ denote the set of nonzero elements of~$M$. The abelian group consisting of all invertible elements of $M$ is denoted by $\uu(M)$, and $M$ is called \emph{reduced} provided that the abelian group $\uu(M)$ is trivial. One can readily check that the quotient $M/\uu(M)$ is a reduced monoid, and we denote this quotient by $M_{\text{red}}$.
\smallskip

For $b,c \in M$, we say that $c$ \emph{divides} $b$ in $M$ if we can write $b = c+d$ for some $d \in M$, in which case we write $c \mid_M b$. Two elements $b,c \in M$ are called \emph{associates} if $b \mid_M c$ and $c \mid_M b$. A \emph{maximal common divisor} of a nonempty subset $S$ of $M$ is a common divisor $d \in M$ such that the only common divisors of the set $\big\{m_s : m_s+d \in S \big\}$ are the invertible elements of~$M$. We say that $M$ is \emph{$2$-MCD} if any two elements of $M$ have a maximal common divisor.
\smallskip

A non-invertible element $a \in M$ is called an \emph{atom} provided that, for all $b,c \in M$, the equality $a = b+c$ implies that either $b \in \uu(M)$ or $c \in \uu(M)$. The set consisting of all atoms of $M$ is denoted by $\mathcal{A}(M)$. If $\mathcal{A}(M)$ is empty, then the monoid $M$ is called \emph{antimatter}. The additive monoid $\qq_{\ge 0}$ is a simple example of an antimatter monoid; indeed, every rational $q \in \qq_{>0}$ is the sum of two copies of $\frac{q}2$ in $\qq_{\ge 0}$, thus $\qq_{\ge 0}$ contains no atoms. An element $q \in M$ is called \emph{atomic} if either $q \in \uu(M)$ or~$q$ can be written as a sum of finitely many atoms of $M$. Following~\cite{pC68}, we say that the monoid $M$ is \emph{atomic} if every element of $M$ is atomic. One can readily check that $M$ is atomic if and only if $M_{\text{red}}$ is atomic. Following \cite{AAZ90}, we say that $M$ is a \emph{strongly atomic} monoid if $M$ is simultaneously atomic and $2$-MCD. See~\cite{CG24} for a recent survey on atomicity in the class of integral domains and~\cite{GZ20} for a survey on factorization theory in the class of commutative monoids. For a generous background on factorization theory in atomic monoids and domains, see~\cite{GH06b}.
\smallskip

Now suppose that $M$ is an atomic monoid. For $q \in M \setminus \uu(M)$ and $a_1, \dots, a_\ell \in \mathcal{A}(M)$ such that $q = a_1 + \dots + a_\ell$, we call the formal sum of (possibly repeated) atoms $a_1 + \dots + a_\ell$ a \emph{factorization} of~$q$. Two factorizations of the same element are considered the same up to permutations of their atoms and replacements of some of their atoms by some of their corresponding associates. We further assume that each element of $\uu(M)$ has a unique factorization, namely, $0$. For each $q \in M$, we let $\mathsf{Z}_M(q)$ denote the set of all factorizations of $q$ in $M$, and we drop the subscript~$M$ from $\mathsf{Z}_M(q)$ when we see no risk of ambiguity (in particular, $\mathsf{Z}_M(u) := \{0\}$ for any $u \in \uu(M)$. We say that $M$ is a \emph{finite factorization monoid} (FFM) if every element of $M$ has only finitely many factorizations, while we say that $M$ is a \emph{unique factorization monoid} (UFM) if every element of $M$ has a unique factorization.

For any $a_1, \dots, a_\ell$, the factorization $z := a_1 + \dots + a_\ell$ is said to have \emph{length} $\ell$, and we often write the length of $z$ as $|z|$. For each $q \in M$, we set
\[
	\mathsf{L}_M(q) := \{|z| : z \in \mathsf{Z}(q)\}.
\]
As for sets of factorizations, we drop the subscript $M$ from $	\mathsf{L}_M(q)$ when there seems to be no risk of ambiguity. The monoid $M$ is called a \emph{bounded factorization monoid} (BFM) if $\mathsf{L}(q)$ is finite for every $q \in M$. It follows directly from the definitions that every FFM is a BFM. See~\cite{AG22} for a survey on the bounded and finite factorization property in the class of integral domains. Another weaker notion of the finite factorization is the length-finite factorization property, which we proceed to introduce. For each $q \in M$ and $\ell \in \nn$, we set
\[
	\mathsf{Z}_\ell(q) := \{z \in \mathsf{Z}(q) : |z| = \ell \}.
\]
Then we say that $M$ is a \emph{length-finite factorization monoid} (LFFM) provided that for each $q \in M$ the set $\mathsf{Z}_\ell(q)$ is finite for every $\ell \in \nn$. It follows directly from the corresponding definitions that every FFM is an LFFM. On the other hand, it is known that the notions of BFM and LFFM are not comparable (we will discuss examples illustrating these observations in Section~\ref{sec:factorization properties}). See~\cite{aG16} for a survey on sets of lengths.

\bigskip
\section{Atomicity}
\label{sec:atomic properties}

In this section, we study how the property of being atomic behaves under internal sum in the class of Puiseux monoids. Let us begin by proving that atomicity and strong atomicity are both preserved under the internal sum with a finitely generated Puiseux monoid.

\begin{theorem} \label{thm:atomic+fg}
	Let $M$ and $N$ be Puiseux monoids such that $N$ is finitely generated. Then the following statements hold.
	\begin{enumerate}
		\item If $M$ is atomic, then $M+N$ is atomic. 
		\smallskip
		
		\item If $M$ is strongly atomic, then $M+N$ is strongly atomic.
	\end{enumerate}
	
\end{theorem}

\begin{proof}
	Since $N$ is finitely generated, we can assume that $N$ is a cyclic Puiseux monoid and then extend to any finitely generated monoid inductively. Let $r$ be a positive rational, and set $S := M + \nn_0 r$.  If $r \in M$, then $S = M$ and so $S$ is atomic (resp., strongly atomic) provided that $M$ is atomic (resp., strongly atomic). 
	Therefore, we assume that $r \notin M$. 
	
	(1) Suppose that $M$ is atomic. Since $\mathcal{A}(M)$ generates~$M$, it follows that $\mathcal{A}(M) \cup \{r\}$ is a generating set of~$S$. Thus, it remains to show that every element in $A := \mathcal{A}(M) \cup \{r\}$ is atomic in~$S$. 
	Because $r$ is the minimum nonzero element of $\nn_0 r$, the fact that $r \notin M$ guarantees that $r \in \mathcal{A}(S)$, and so $r$ is atomic. To argue that every atom of $M$ is atomic in $S$, fix $a \in \mathcal{A}(M)$. If~$a$ is an atom in $S$, then there is nothing to do. Therefore we assume that $a$ is not an atom in~$S$. Thus, $r$ properly divides $a$ in $S$. Let $m$ be the largest positive integer such that $mr \mid_S a$. Then $a - mr \in M$, and so the fact that $M$ is atomic ensures the existence of $a_1, \dots, a_k \in \mathcal{A}(M)$ such that $a = (a_1 + \dots + a_k) + mr$. It follows now from the maximality of $m$ that $r \nmid_S a_i$ for any $i \in \ldb 1,k \rdb$; hence, $a_1, \dots, a_k$ remain atoms in~$S$. This, along with the fact that $r$ is atomic in $S$, guarantees that $a$ is atomic in~$S$. Hence the Puiseux monoid $S$ is atomic.
	\smallskip
	
	(2) Now suppose that $M$ is strongly atomic. As $M$ is atomic, it follows from part~(1) that $S$ is also atomic. Thus, it is enough to prove that $S$ is a $2$-MCD monoid, that is, any two elements of $S$ have a maximal common divisor. As before, the fact that $r \notin M$ ensures that $r \in \mathcal{A}(S)$. For each $c \in S$, we let $m_c$ denote the largest nonnegative integer such that $m_c r \mid_S c$, which must exist as $r$ is positive. Using this notation, each $c \in S$ can be written as $c = b + m_cr$ for some $b \in S$ such that $r \nmid_S b$.
	\smallskip
	
	\noindent \textsc{Claim.} For any $x,y \in S$ with $m_x = 0$ the set $\{x,y\}$ has a maximal common divisor in~$S$. 
	\smallskip
	
	\noindent \textsc{Proof of Claim.} We proceed by induction on $m_y$. If $m_y = 0$, then all the divisors of $x$ and $y$ belong to $M$, so the fact that $M$ is a $2$-MCD guarantees that $\{x,y\}$ has a maximal common divisor in $M$, and thus in $S$ as well. 
	Now suppose that $m_y > 0$, and suppose that the statement of the claim holds for any subset $\{x',y'\}$ of $S$ with $m_{x'} = 0$ and $m_{y'} < m_y$. Write $y = y' + m_yr$ for some $y' \in S$ such that $r \nmid_S y'$. Then each divisor of $y'$ in $S$ is also a divisor of $y'$ in $M$. Let $d_1$ be a maximal common divisor of $x$ and $y'$ in $M$, which must exist because $M$ is a $2$-MCD monoid. If $x - d_1$ and $y - d_1$ have no non-invertible (nonzero in the context of Puiseux monoids) common divisor in $S$, then $d_1$ is a maximal common divisor of $x$ and $y$ in $S$, and we are done. Otherwise, let $d_2$ be a nonzero common divisor of $x-d_1$ and $y-d_1$ in $S$. We know that $d_2 \nmid_M y' - d_1$; thus, because each divisor of $y'$ in $S$ must belong to~$M$, we must have $d_2 \nmid_S y' - d_1$. This means that $m_{y-d_1-d_2} < m_y$, since
	\[
		y - d_1 - d_2 - m_y r = y' - d_1 - d_2 \notin S.
	\]
	It follows now from our induction hypothesis that the set $\{x - d_1 - d_2, y - d_1 - d_2\}$ has a maximal common divisor in $S$, implying that $\{x,y\}$ has a maximal common divisor in~$S$. Hence the claim is established.
	
	Now consider any $x,y$ in $S$ with $m_x > 0$. Without loss of generality, we can assume that $m_x \le m_y$. Because  $m_{x - m_x r} = 0$, it follows from our established claim that the set $\{x - m_x r, y - m_x r\}$ has a maximal common divisor in $S$. Thus, $\{x, y\}$ also has a maximal common divisor in $S$. Hence, we conclude that the Puiseux monoid $S$ is strongly atomic.
\end{proof}

%
%
%

However, if we consider the abelian group $\zz^2$ as our universe, it is not true that the internal sum of a strongly atomic submonoid of $\zz^2$ and a finitely generated submonoid of $\zz^2$ is an atomic monoid. The following example sheds some light upon this observation. 

\begin{example} \label{ex:sum of rank-2 atomic lattice monoids that is not atomic}
	We describe two atomic submonoids $M_1$ and $M_2$ of the free abelian group $\zz^2$ whose internal sum in $\zz^2$ is not atomic, even though $M_1$ is finitely generated. Let $M_1$ be the monoid consisting of all the lattice points in the first quadrant of $\rr^2$, that is, $M_1 := \nn_0 \times \nn_0$. Observe that $M_1$ is the free commutative monoid of rank $2$ and, therefore, $M_1$ is atomic with $\mathcal{A}(M_1) = \{e_1, e_2\}$, where $e_1 := (1,0)$ and $e_2 := (0,1)$. In particular, $M_1$ is generated by the finite set $\{e_1, e_2\}$. Now suppose that $M_2$ is the submonoid of $\zz^2$ whose set of nonzero elements are all the lattice points strictly above the $x$-axis, that is $M_2 := (0, 0) \cup \mathbb Z \times \mathbb N$. One can readily check that $\mathcal{A}(M_2) = \{(n,1) : n \in \zz\}$, which immediately implies that $M_2$ is also atomic. In fact, if we take any two $u_1:=(x_1,y_1),u_2:=(x_2,y_2)\in M_2$ and assume without loss of generality that  $y_1<y_2$, then $u_1$ is a maximal common divisor of $u_1$ and $u_2$, because it is clear that $u_1\mid_{ M_2} u_2$ and $u_1-u_1=0$ do not have reducible divisors. This means $M_2$ is strongly atomic. Now let $M$ be the internal sum in $\zz^2$ of $M_1$ and $M_2$; that is, $M := M_1 + M_2$. It is clear that $M_1 \cup M_2 \subseteq M$. On the other hand, for each nonzero $v := (x,y)\in M_2$, the fact that $y \in \zz_{\ge 1}$ ensures that the inclusion $M_1 + v \subseteq M_2$ holds. Therefore
    \[
        M = M_1 \cup \big( \bigcup_{v \in M_2^\bullet} (M_1 + v) \big) \subseteq M_1 \cup M_2.
    \]
    Hence $M = M_1 \cup M_2$, from which we see that $M$ is the nonnegative cone of $\zz^2$ under the lexicographical order with priority on the second coordinate. As a result, the only atom of $M$ is the minimum of $M^\bullet$, which means that $\mathcal{A}(M) = \{e_1\}$. As a consequence, $M$ is not atomic: indeed, the set of atomic elements of $M$ is $\{(m,0) : m \in \nn_0\}$, which is strict subset of $M$.
\end{example}

We conclude this section, showing that the internal sum inside $\qq$ of two atomic Puiseux monoids may not be atomic. Indeed, we will provide a more drastic construction: two atomic Puiseux monoids whose internal sum is antimatter.

\begin{example}
	Let us construct two atomic Puiseux monoids $M_1$ and $M_2$ such that $M_1+M_2$ is antimatter. Let $p_n$ denote the $n$-th odd prime, and set $a_n := \frac{1}{2^np_n}$ for every $n \in \nn$. Now set $S_1 := \{a_n \colon n \in \mathbb N\}$, and let $M_1$ be the Puiseux monoid generated by $S_1$. 
	It is routine to argue that $M_1$ is atomic with $\mathcal{A}(M_1) = S_1$ (the Puiseux monoid $M_1$ is the main ingredient in Grams's construction of the first atomic domain not satisfying the ACCP, and due to this historical fact $M_1$ is often called Grams' monoid).
 
	Let $f \colon \mathbb N \to \mathbb N$ be an injective function such that $p_{f(n)} > 2^np_n$ for all $n \in \mathbb N$. Now for each $n \in \mathbb N$, we construct an infinite sequence $(b_i)_{i \ge 1}$ such that
	\[
		b_1 := a_n - a_{f(n)} > \frac{1}{2}a_n \quad \text{and} \quad b_{i+1} := b_i - \frac{1}{2^{c_i}},
	\]
	where $c_i \in \mathbb N$ is chosen such that $b_{i+1} > \frac12 a_n$. We will refer to the elements in $(b_i)_{i \ge 1}$ as \emph{$n$-atoms}. Let $S_2$ be the union of the sets of $n$-atoms over all $n$, and let $M_2$ be the Puiseux monoid generated by $S_2$.
	\smallskip

    \noindent \textsc{Claim.} $M_2$ is atomic with $\mathcal{A}(M_2) = S_2$.
    \smallskip
    
	\noindent \textsc{Proof of Claim.} Suppose by way of contradiction that there exists an element in $S_2$ that can be expressed as the finite sum of other elements in $S_2$. Suppose that this element is an $n$-atom, $b_k$, for some $n \in \mathbb N$. Then $$b_k = \sum_{s \in S'} s$$ for some multiset $S' \subset S_2\setminus \{b_k\}$ with $|S'| < \infty$.
		
	Because $b_k$ has a factor of $p_n$ in its denominator, $S'$ must also contain elements with a factor of $p_n$ in their denominators. The only such elements in $S_2$ are either $n$-atoms or $m$-atoms for the at most one value of $m$ such that $f(m) = n$. We split the rest of the argument into two cases.
    \smallskip

	\noindent \textit{Case i:} If $S'$ contains the $n$-atom $b_i$, note that we can have at most one $n$-atom in $S'$, since $b_i + b_j \ge 2b_j > b_1 \ge b_k$ for all $j \ge i$. We can write $$b_k-b_i = \sum_{s \in S'\setminus \{b_i\}} s,$$ where $b_k-b_i$ is a sum of powers of $\frac{1}{2}$. Furthermore, $b_k - b_i < b_i \le b_1 < \frac{1}{2}$. 
			
	Let $\ell \in \mathbb N$ be such that $S'$ contains $\ell$-atoms but does not contain $f(\ell)$-atoms; this is possible because $S'$ is finite. Since the denominator of $b_k - b_i$ is not divisible by any odd primes, there must exist at least $p_{f(\ell)}$ $\ell$-atoms in $S'$. Since all $\ell$-atoms are greater than $\frac12 a_\ell$, the total sum of these $\ell$-atoms in $S'$ is greater than $p_{f(\ell)} \cdot \frac12 a_\ell = \frac12 \left(\frac{p_{f(\ell)}}{2^\ell p_\ell}\right) > \frac12$ (recall that $p_{f(\ell)} > 2^\ell p_\ell$). Thus $$\sum_{s \in S'\setminus \{b_i\}} s > \frac12 > b_k-b_i,$$ a contradiction.
    \smallskip
    
	\noindent \textit{Case ii:} Otherwise, we must have at least $p_n -1$ $m$-atoms in $S'$. Note that $$\sum_{s \in S'} s = b_k \le b_1 = a_n - a_{f(n)} \le a_n < \frac12.$$ However, since all $m$-atoms are greater than $\frac12 a_m$, the sum of these $p_n -1$ $m$-atoms in $S'$ is already greater than $\frac12 \left(\frac{p_n-1}{2^mp_m}\right) \ge \frac12$. This dictates that $$\sum_{s \in S'} s > \frac12,$$ a contradiction. 
	\smallskip
 
	Since both cases yield a contradiction, all elements of $S_2$ must be atoms. Hence $M_2$ is atomic, and our claim is established.
	\smallskip
 
	We can finally argue that the Puiseux monoid $M := M_1 + M_2$ is antimatter despite being the internal sum of two atomic Puiseux monoids. To do so, first observe that any $a_n \in S_1$, which is an atom of $M_1$, is divisible in $M$ by any $n$-atom and therefore cannot be an atom of $M$. Also, any $n$-atom $b_k \in S_2$, which is an atom of $M_2$, is divisible in $M$ by $b_{k+1}$, which means it cannot be an atom of $M$. Thus, none of the generators of $M$ is irreducible, which clearly implies that the Puiseux monoid $M$ is antimatter.
\end{example}

\bigskip
\section{The Bounded and Finite Factorization Properties} 
\label{sec:factorization properties}

As the following theorem indicates, inside the class of Puiseux monoids the properties of being an FFM, BFM, or LFFM are preserved under the internal sum with a finitely generated monoid (in Example~\ref{ex:sum of FFMs is not FFM in Z^2} we show that the same properties are not preserved under the internal sum with a finitely generated monoid in the class consisting of all submonoids of $\zz^2$). 

\begin{theorem} \label{thm:ff/lff/bf + fg}
	Let $M$ and $N$ be Puiseux monoids such that $N$ is finitely generated, then the following statements hold.
	\begin{enumerate}
		\item If $M$ is an FFM, then $M+N$ is an FFM.
		\smallskip
		
		\item If $M$ is a BFM, then $M+N$ is a BFM.
		\smallskip
		
		\item If $M$ is an LFFM, then $M+N$ is an LFFM.
	\end{enumerate} 
\end{theorem}

\begin{proof}
	Since $N$ is finitely generated, it suffices to consider the case when $N$ is a cyclic monoid. To do so, let $r$ be a positive rational and consider the Puiseux monoid $S := M + \nn_0 r$. In light of Theorem~\ref{thm:atomic+fg}, the monoid $S$ is atomic if $M$ is either an FFM, BFM, or LFFM. Observe that if $r \in M$, then $S = M$ and so $S$ is an FFM (resp., a BFM or an LFFM) if and only if $M$ is an FFM (resp., a BFM or an LFFM). Therefore we assume that $r \notin M$.
	\smallskip
	
	(1) Suppose that $M$ is an FFM, and let us argue that $S$ is an FFM. As $S$ is atomic, it suffices to show that every element of $S$ has only finitely many factorizations. Fix a nonzero $s \in S$. Since $r \notin M$, the fact that $r$ is the minimum of $\nn r$ implies that $r \in \mathcal{A}(S)$. Now notice that there are only finitely many $c \in \nn_0$ such that $s - cr \in M$: let them be $c_1, \dots, c_n$. Since every factorization of $s$ in $S$ has the form $(a_1 + \dots + a_m) + c_k r$, it follows that
	\begin{equation} \label{eq:factorization set}
		\mathsf{Z}_S(s) = \bigcup_{k=1}^n \big( \mathsf{Z}_M(s - c_k r) + c_k r \big).
	\end{equation}
	Now the fact that $M$ is an FFM ensures that $|\mathsf{Z}_M(s - c_k r)| \le \infty$ for every $k \in \ldb 1,n \rdb$. As a consequence, $\mathsf{Z}_S$(s) is the finite union of finite sets, and therefore is itself finite. Hence we conclude that the Puiseux monoid $S$ is an FFM. 
	\smallskip
	
	(2) Now suppose that $M$ is a BFM, and let us show that $S$ is also a BFM. To do so, fix a nonzero $s \in S$. As in the proof of part~(1), we can see that $r \in \mathcal{A}(S)$ and we can check that there are only finitely many $c \in \nn_0$ such that $s - cr \in M$, which we denote by $c_1, \dots, c_n$. Observe that the equality~\eqref{eq:factorization set} still holds, and it implies that $\mathsf{L}_S(s) = \bigcup_{k=1}^n \big( \mathsf{L}_M(s - c_k r) + c_k \big)$. Now the fact that $M$ is a BFM ensures that $|\mathsf{L}_M(s - c_k r)| \le \infty$ for every $k \in \ldb 1,n \rdb$, and this implies that the set $\mathsf{L}_S(s)$ is finite. Hence the Puiseux monoid $S$ is a BFM.
	\smallskip
	
	(3) Finally, suppose that $M$ is an LFFM. Assume, for the sake of contradiction, that there exist $\ell \in \nn$ and an element $q \in S$ with infinitely many factorizations of length~$\ell$. As in the previous two parts, $r \in \mathcal{A}(S)$. Since $M$ is an LFFM, the element $r$ must appear in infinitely many of these length-$\ell$ factorizations. However, $r>0,$ and so $r$ can only divide $n$ a finite number of times $f$. As a result, there must be an $i \in \nn$ with $i \leq f$ and $0<i<\ell$ such that infinitely many length-$\ell$ factorizations have exactly $i$ occurrences of the factor $r$ and the other factors are all in $M.$ This implies that $q - ir\in M$ has infinitely many factorizations of length $\ell-i$, which is a contradiction.
\end{proof}

In order to argue that each statement in Theorem~\ref{thm:ff/lff/bf + fg} are sharp, we proceed to construct two Puiseux monoids that are FFMs whose internal sum is not even an LFFM, and then we construct two Puiseux monoids that are BFMs whose internal sum is not even atomic.

\begin{example} \label{ex:sum of two FFMs not even LFF}
	For each $n \in \nn$, we let $p_n$ denote the $n$-th prime number in the set $\pp_{\ge 5}$. Now we consider the sequences $(a_n)_{n \ge 1}$ and $(b_n)_{n \ge 1}$ whose terms are the positive rationals defined as follows:
	\[
		a_n :=  \frac{p_n -1}{p_n} \text{ and } \quad b_n := \frac{p_n + 1}{p_n}
	\]
	for every $n \in \nn$. Now consider the Puiseux monoids $M_A$ and $M_B$ respectively generated by the sets $A := \{a_n : n \in \nn\}$ and $B := \{b_n : n \in \nn\}$. It is routine to verify that $\mathcal{A}(M_A) = A$ and $\mathcal{A}(M_B) = B$, and so $M_A$ and $M_B$ are atomic monoids.
	
	Let us show that both $M_A$ and $M_B$ are FFMs. Because the sequence $(a_n)_{n \ge 1}$ is an increasing sequence, it follows from \cite[Theorem~5.6]{fG19} that $M_A$ is an FFM. As $M_B$ is an atomic reduced monoid, it follows from \cite[Proposition~3.6]{AG22} that $M_B$ is an FFM if and only if every nonzero element of $M_B$ is divisible by only finitely many atoms. To argue this equivalent condition, let $q \in M_B$ be a nonzero element. Take $N \in \nn$ such that $p_n \nmid \mathsf{d}(q)$ and $p_n > q$ for any $n \ge N$.
	\smallskip
	
	\noindent \textsc{Claim.} $b_n \nmid_M q$ for any $n \ge N$. 
	\smallskip
	
	\noindent \textsc{Proof of Claim.} Suppose, by way of contradiction, that $b_k \mid_M q$ for some $k \ge N$. In this case, the fact that $M_B$ is atomic with $\mathcal{A}(M_B) = \{b_n : n \in \nn\}$ allows us to write
	\begin{align} \label{eq:aux q}
		q = \sum_{n=1}^\ell c_n \frac{p_n + 1}{p_n}
	\end{align}
	for some coefficients $c_1, \dots, c_\ell \in \nn_0$ such that $k \le \ell$ and $c_k > 0$. Since $k \ge N$, it follows that $p_k \nmid \mathsf{d}(q)$. Therefore the $p_k$-adic valuation of $q$ is nonnegative and so, after applying the $p_k$-adic valuation map to both sides of~\eqref{eq:aux q}, we obtain that $p_k \mid c_k$. Thus, $q \ge \frac{c_k}{p_k}(p_k + 1) \ge p_k + 1$, which contradicts the fact that $k \ge N$. Hence the claim is established.
	\smallskip
	
	From the established claim, we obtain that $q$ is only divisible by some of the atoms $b_1, \dots, b_{N-1}$, and so the Puiseux monoid $M_B$ is also an FFM.
	\smallskip
	
	Finally, consider the Puiseux monoid $M := M_A + M_B$. It is clear that $\mathcal{A}(M) \subseteq \mathcal{A}(M_A) \cup \mathcal{A}(M_B) = \{a_n, b_n : n \in \nn\}$. For each $n \in \nn$, the fact that $p_n \in \pp_{\ge 5}$ implies that $\frac34 \lneq a_n \lneq 1$ and $1\lneq b_n \lneq \frac54$. This, along with the fact that all the atoms of $M$ belong to $\{a_n, b_n : n \in \nn\}$, guarantees that
	\[
		\mathcal{A}(M) =  \{a_n, b_n : n \in \nn \} = \Big\{ \frac{p_n \pm 1}{p_n} : n \in \nn \Big\}.
	\]
	Also, as $0$ is not a limit point of $\mathcal{A}(M)$, it follows from \cite[Proposition~4.5]{fG19} that $M$ is a BFM. In particular, $M$ is atomic. On the other hand, the equality
	\[
		2 =  \frac{p_n -1}{p_n} + \frac{p_n + 1}{p_n} = a_n + b_n
	\]
	holds for every $n \in \nn$, which implies that $2$ has infinitely many length-$2$ factorizations in $M$. As a consequence, the monoid $M$ is not an FFM. Hence $M_A$ and $M_B$ are two FFMs whose internal sum is not an FFM.

\end{example}

Now we illustrate that the internal sum of two Puiseux monoids that are BFMs may not be a BFM. This example is motivated by \cite[Example~3.4]{GG24}, where the authors show that the internal sum of two Puiseux monoids satisfying the ascending chain condition on principal ideals may not satisfy the same condition.

\begin{example} \label{ex:sum of BFMs that is not a BFM}
	We will produce two Puiseux monoids $M_1$ and $M_2$ that are BFMs such that their sum $M_1 + M_2$ is not even atomic. Set $M_1 := \{0\} \cup \qq_{\ge 1}$. Because $0$ is not a limit point of $M_1^\bullet$, it follows from \cite[Proposition~4.5]{fG19} that $M_1$ is a BFM. Now consider the Puiseux monoid $M_2$ defined as follows:
	\[
		M_2 := \bigg\langle \frac{p_n + 1}{p_n^2} : n \in \nn \bigg\rangle,
	\]
	where $p_n$ is the $n$-th prime. It is routine to check that $\mathcal{A}(M_2) = \big\{ \frac{p_n  +1}{p_n^2} : n \in \nn \big\}$, which immediately implies that $M_2$ is an atomic monoid. Moreover, we can actually prove that $M_2$ is a FFM by mimicking the argument used to show that $M_B$ is a FFM in Example~\ref{ex:sum of two FFMs not even LFF}. Thus, $M_2$ is a BFM. 
	
	Now let $M$ be the internal sum of $M_1$ and $M_2$ inside $\qq$; that is, $M = M_1 + M_2$. We claim that $M$ is not atomic. First, note that because $\min M_1^\bullet = 1 > \sup \mathcal{A}(M_2)$, no nonzero element of $M$ can divide any of the atoms of $M_2$ in $M$, which implies that $\mathcal{A}(M_2) \subseteq \mathcal{A}(M)$. On the other hand, we observe that if we take $q \in M_1$ with $q > 1$, the fact that $\lim_{n \to \infty} \frac{p_n + 1}{p_n^2} = 0$ allows us to take $n \in \nn$ large enough so that $\frac{p_n + 1}{p_n^2} < q-1$, in which case the equality
	\[
		q = \bigg( q - \frac{p_n + 1}{p_n^2} \bigg) + \frac{p_n + 1}{p_n^2} \in M^\bullet + M^\bullet
	\]
	ensures that $q \notin \mathcal{A}(M)$. Hence $\mathcal{A}(M) \subseteq \{1\} \cup \big\{ \frac{p_n + 1}{p_n^2} : n \in \nn \big\}$. This in turn implies that the element $\frac98 \in M$ is not atomic because if $r \in \big\langle \{1\} \cup \big\{ \frac{p_n + 1}{p_n^2} : n \in \nn \big\} \big\rangle$, then the denominator of $r$ cannot be divisible by the cube of any prime. As a result, the Puiseux monoid $M_1 + M_2$ is not atomic (and so not a BFM), even though $M_1$ and $M_2$ are BFMs.
\end{example}

For torsion-free abelian groups of rank larger than $1$, we can find submonoids that are FFMs whose internal sum is not even atomic. To illustrate this, we revisit Example~\ref{ex:sum of rank-2 atomic lattice monoids that is not atomic}, where we exhibited two atomic rank-$2$ submonoids of the abelian group $\zz^2$ whose sum is not atomic.

\begin{example} \label{ex:sum of FFMs is not FFM in Z^2}
	Let $M_1$ and $M_2$ be the rank-$2$ submonoids of the free abelian group $\zz^2$. Since $M_1$ is the free commutative monoid of rank-$2$, it is a UFM and so it must be an FFM. On the other hand, observe that the reduced monoid $M_2/\uu(M_2)$ is isomorphic to the free commutative monoid $\nn_0$; thus, it is a UFM and, therefore, an FFM. Thus, $M_1$ and $M_2$ are both FFMs, and we have already see in Example~\ref{ex:sum of rank-2 atomic lattice monoids that is not atomic} that $M_1 + M_2$ is the nonnegative cone of $\zz^2$ under the lexicographical order (on the second coordinate), and so it is not even atomic. 
\end{example}

\bigskip
\section{Final Remarks on the Main Results}
\label{sec:final remarks}

In our main results, Theorems~\ref{thm:atomic+fg} and~\ref{thm:ff/lff/bf + fg}, we have chosen finitely generated monoids because they are some of the most standard and basic monoids studied in monoid and semigroup theory. One may wish to extend both theorems by replacing finitely generated monoids by a larger and natural class of monoids. It follows from \cite[Proposition~2.7.8]{GH06b} that every finitely generated monoid is an FFM, and so a BFM. On the other hand, it follows from~\cite[Proposition~4.5]{fG19} that every Puiseux monoid $M$ is a BFM provided that $0$ is not a limit point of $M^\bullet$. We call a Puiseux monoid $M$ a \emph{bounded below monoid} (BBM) if $0$ is not a limit point of $M^\bullet$. Therefore every Puiseux monoid that is a BBM is also a BFM, and so we obtain the following diagram of implications in the class of Puiseux monoids:
\begin{equation*}
    \begin{tikzcd}
        \textbf{ FGM } \ \arrow[r, Rightarrow]  \arrow[d, Rightarrow] & \ \textbf{ BBM } \arrow[d, Rightarrow] \\
        \textbf{ FFM } \ \arrow[r, Rightarrow] & \ \textbf{ BFM }
    \end{tikzcd}
\end{equation*}
where FGM stands for ``finitely generated monoid". Observe that the Puiseux monoids $M_A$ and $M_B$ in Example~\ref{ex:sum of two FFMs not even LFF} are both FFMs and BBMs but they are not finitely generated. In addition, the Puiseux monoid $M_2$ in Example~\ref{ex:sum of BFMs that is not a BFM} is an FFM (and so a BFM) that is not a BBM, while the Puiseux monoid $\{0\} \cup \qq_{\ge 1}$ is a BBM that is not an FFM (and so not a LFFM) because $|\mathsf{Z}(3)| = \infty$ (indeed, since the set of atoms of $\{0\} \cup \qq_{\ge 1}$ is $\qq \cap [1,2)$, it follows that for each $n \ge 3$ the equality $3 = \big( \frac32 - \frac1n \big) + \big( \frac32 + \frac1n \big)$ yields a length-$2$ factorization of $3$).

We conclude with the following remark on some potential natural ways to strengthen the main theorems we have established in this paper.

\begin{remark}
    We cannot strengthen Theorems~\ref{thm:atomic+fg} and~\ref{thm:ff/lff/bf + fg} by replacing the condition that $N$ is finitely generated by the condition that $N$ is a BBM or by the condition that $N$ is an FFM. 
    \begin{itemize}
        \item Checking this for Theorem~\ref{thm:atomic+fg} amounts to observing that the Puiseux monoid $M_2$ in Example~\ref{ex:sum of BFMs that is not a BFM} is an FFM that is not a BBM (in particular, $M_2$ is strongly atomic) while the Puiseux monoid $\{0\} \cup \qq_{\ge 1}$ is a BBM that is not an FFM (in particular, $\{0\} \cup \qq_{\ge 1}$ is strongly atomic), but their internal sum is not even atomic, as argued in Example~\ref{ex:sum of BFMs that is not a BFM}.
        \smallskip

        \item The previous observation also guarantees that we cannot replace the condition that $N$ is finitely generated by the condition that $N$ is either a BBM or an FFM in parts~(2) and (3) of Theorem~\ref{thm:ff/lff/bf + fg}. The same observation ensures that we cannot replace the condition that $N$ is finitely generated by the condition that $N$ is a BBM in part~(1) of the same theorem. Finally, in light of Example~\ref{ex:sum of two FFMs not even LFF}, where we exhibited two Puiseux monoids that are FFMs whose internal sum is not an FFM, we can conclude that we cannot replace the condition that $N$ is finitely generated by the condition that $N$ is an FFM to strengthen part~(1) of Theorem~\ref{thm:ff/lff/bf + fg}.
    \end{itemize} 
\end{remark}

\bigskip
\section*{Acknowledgments}
The authors would like to kindly thank our PRIMES-USA mentor, Dr. Felix Gotti, for his instruction, guidance, and continued support throughout our research. We would also like to thank the staff at MIT-PRIMES for all of the time and effort they have put in to provide this nonpareil opportunity to conduct math research.

\bigskip


\begin{thebibliography}{20}
	
	\bibitem{AAZ90} D.~D.~Anderson, D.~F.~Anderson, and M.~Zafrullah, \emph{Factorizations in integral domains}, J. Pure Appl. Algebra \textbf{69} (1990) 1--19.

	\bibitem{AG22} D. F. Anderson and F. Gotti, \emph{Bounded and finite factorization domains}. In: Rings, Monoids, and Module Theory (Eds. A. Badawi and J. Coykendall) pp. 7--57, Springer Proceedings in Mathematics \& Statistics, vol. 382, Singapore, 2022.
		
	\bibitem{CGG20a} S.~T. Chapman, F. Gotti, and M. Gotti, \emph{When is a Puiseux monoid atomic?}, Amer. Math. Monthly \textbf{128} (2021) 302--321.
	
	\bibitem{pC68} P. M. Cohn, \emph{Bezout rings and their subrings}, Proc. Cambridge Philos. Soc. \textbf{64} (1968) 251--264.

    \bibitem{CG24} J. Coykendall and F. Gotti, \emph{Atomicity in integral domains}. In: Rings and Factorizations (Eds. M. Brešar, A. Geroldinger, B. Olberding, and D. Smertnig), Springer Nature, Switzerland, 2024.

    \bibitem{lF70} L. Fuchs, Infinite Abelian Groups I, Academic Press, 1970.

	\bibitem{aG16} A. Geroldinger, \emph{Sets of lengths}, Amer. Math. Monthly \textbf{123} (2016) 960--988.
 
	\bibitem{GG24} A. Geroldinger and F. Gotti, \emph{On monoid algebras having every nonempty subset of $\mathbb{N}_{\ge 2}$ as a length set}. Submitted. Preprint on arXiv: https://arxiv.org/pdf/2404.11494
	
	\bibitem{GGT21} A. Geroldinger, F. Gotti, and S. Tringali, \emph{On strongly primary monoids, with a focus on Puiseux monoids}, J. Algebra \textbf{567} (2021) 310--345.
	
	\bibitem{GH06b} A.~Geroldinger and F.~Halter-Koch, \emph{Non-Unique Factorizations: Algebraic, Combinatorial and Analytic Theory}, Pure and Applied Mathematics Vol. 278, Chapman \& Hall/CRC, Boca Raton, 2006.

    \bibitem{GZ21} A. Geroldinger and Q. Zhong, \emph{A characterization of length-factorial Krull monoids}, New York J. Math. \textbf{27} (2021) 1347--1374.

    \bibitem{GZ20} A. Geroldinger and Q. Zhong, \emph{Factorization theory in commutative monoids}, Semigroup Forum \textbf{100} (2020) 22–-51.

	\bibitem{rG84} R.~Gilmer, \emph{Commutative Semigroup Rings}, Chicago Lectures in Mathematics, The University of Chicago Press, London, 1984.

	\bibitem{fG20} F. Gotti, \emph{Geometric and combinatorial aspects of submonoids of a finite-rank free commutative monoid}, Linear Algebra Appl. \textbf{604} (2020) 146--186.

    \bibitem{fG18} F. Gotti, \emph{Puiseux monoids and transfer homomorphisms}, J. Algebra \textbf{516} (2018) 95--114.
    
	\bibitem{fG19} F.~Gotti, \emph{Increasing positive monoids of ordered fields are FF-monoids}, J. Algebra \textbf{518} (2019) 40--56.

	\bibitem{aG74} A.~Grams, \emph{Atomic rings and the ascending chain condition for principal ideals}. Math. Proc. Cambridge Philos. Soc. \textbf{ 75} (1974) 321--329.

	\bibitem{fHK92} F. Halter-Koch, \emph{Finiteness theorems for factorizations}, Semigroup Forum \textbf{44} (1992) 112--117.

    \bibitem{fHK98} F. Halter-Koch, \emph{Ideal Systems: An Introduction to Multiplicative Ideal Theory}, Marcel Dekker, 1998.
	
\end{thebibliography}
\end{document}